\documentclass[10pt]{amsart}

\usepackage[final]{microtype}
\usepackage{amsmath,amsthm,amssymb,amsfonts}
\usepackage{mathtools}
\usepackage{tikz}
\usetikzlibrary{arrows.meta, positioning}
\usepackage{enumerate}

\usepackage[T1]{fontenc}
\usepackage[utf8]{inputenc}
\usepackage{fourier}

\theoremstyle{plain}
\newtheorem{thm}{Theorem}
\newtheorem{prop}[thm]{Proposition}
\newtheorem{problem}[thm]{Problem}
\newtheorem{question}[thm]{Question}

\newtheorem*{thmA}{Theorem A}
\newtheorem*{propB}{Proposition B}
\newtheorem*{thmB}{Theorem A'}
\newtheorem*{cl}{\textbf{Claim}}

\theoremstyle{definition}
\newtheorem{defn}[thm]{Definition}

\theoremstyle{remark}
\newtheorem{rem}[thm]{\textbf{Remark}}

\newcommand{\bsm}{\begin{smallmatrix}}
\newcommand{\esm}{\end{smallmatrix}}

\newcommand{\R}{\mathbb{R}}
\newcommand{\Z}{\mathbb{Z}}
\newcommand{\PP}{\textrm{Per}(\phi^A,M_A)}
\newcommand{\T}{\mathbb{T}}
\newcommand{\D}{\mathbb{D}}
\newcommand{\s}{\mathbb{R}/\mathbb{Z}}

\newcommand{\SL}{\mathrm{SL}}
\newcommand{\GL}{\mathrm{GL}}

\DeclareMathOperator{\tr}{\mathrm{tr}}
\DeclareMathOperator{\per}{\mathrm{per}}
\DeclareMathOperator{\surg}{\mathrm{Surg}}
\DeclareMathOperator{\length}{\mathrm{length}}

\newcommand{\Fs}{\mathcal{F}^\mathrm{s}}
\newcommand{\Fu}{\mathcal{F}^\mathrm{u}}
\newcommand{\Fcs}{\mathcal{F}^\mathrm{cs}}
\newcommand{\Fcu}{\mathcal{F}^\mathrm{cu}}
\newcommand{\Es}{E^\mathrm{s}}
\newcommand{\Ec}{E^\mathrm{c}}
\newcommand{\Eu}{E^\mathrm{u}}

\title{Infinitely many closed paths in the graph of Anosov flows}
\author{Mario Shannon}
\date{}

\address{Facultad de Ingenier\'ia\\
Universidad de la Rep\'ublica\\
Montevideo, Uruguay}


\begin{document}

\begin{abstract}
Given an Anosov flow on a closed 3-manifold, we are interested in the problem of whether or not making non-trivial Fried surgeries along a finite set of periodic orbits can produce a flow equivalent to itself. We show that for some suspension Anosov flows, there exist infinitely many pairs of periodic orbits satisfying this property.
\end{abstract}

\maketitle

%%%%%%%%%%%%%%%%%%%%%%%%%%%%%%%%%%%%%%%%%%%%%%%%%%%%%%%%%%%%%%%%%%%%%%%%%%%%%%%%%%%%%%%%%%%%%%%%%%%%%%%%%%%%%%%%%%%%%%%%%%%%%%%%%%%%%%%%%%%%%%%%%%
\section{Introduction}
Given an \emph{Anosov flow} on a closed 3-manifold and a \emph{periodic orbit}, there is an operation called \emph{Fried surgery} that allows to construct a new Anosov flow on a new 3-manifold, by making a modification of the initial flow around the chosen periodic orbit. This modification is essentially a \emph{Dehn surgery} along the periodic orbit, but adapted to the pairs (Anosov flow, 3-manifold) in such a way that the resulting manifold is also equipped with an Anosov flow. 

As a consequence, this operation allows to define a \emph{graph structure} on the set of (\emph{orbital equivalence classes} of) Anosov flows, where edges connecting two pairs (Anosov flow, 3-manifold) represent Fried surgeries that allow to obtain one flow from the other. 

The purpose of this note is to show that, at some vertices of the graph corresponding to suspension Anosov flows, there exist infinitely many closed paths of length 2. This replies to a statement made in \cite{Bonatti-Iakovoglou_surgeries} by Bonatti and Iakovoglou, where those authors claim that there should exist at most finitely many loops of this kind.

%%%%%%%%%%%%%%%%%%%%%%%%%%%%%%%%%%%%%%%%%%%%%%%%%%%%%%%%%%%%%%%%%%%%%%%%%%%%%%%%%%%%%%%%%%%%%%%%%%%%%%%%%%%%%%%%%%%%%%%%%%%%%%%%%%%%%%%%%%%%%%%%%%
\subsection{Fried Surgery on Anosov flows}  
A non-singular flow $\phi=\{\phi_t\}_{t\in\R}$ on a closed Riemannian 3-manifold $M$ is \emph{Anosov} if it is generated by a smooth vector field $X$ and the action $D\phi_t:TM\to TM$ on the tangent bundle preserves a \emph{uniformly hyperbolic splitting} $TM=\Es\oplus\Ec\oplus\Eu$, where the $\Es$ (resp. $\Eu$) line-bundle is exponentially contracted (resp. expanded) under the $D\phi_t$-action for $t\geq 0$, and $\Ec$ is a line-bundle colinear with $X$.

We refer to \cite{Katok-Hasselblatt} for general properties about dynamics preserving a uniformly hyperbolic splitting. Among many interesting properties, Anosov flows always have infinitely many \emph{periodic orbits}, all of them \emph{saddle type hyperbolic}. The \emph{stable manifold theorem} asserts that the plane bundles $\Es\oplus\Ec$ and $\Ec\oplus\Eu$ are tangent to a pair of transverse codimension-one foliations $\Fcs$ and $\Fcu$ on $M$, that intersect along the $\phi$-orbits. We call them \emph{center-stable} and \emph{center-unstable} foliations. Every leaf of these foliations is homeomorphic to either a cylinder or a M\"obius  band (in case the leaf contains a periodic orbit), or to a plane.   

Let $\gamma$ be a periodic orbit of an Anosov flow $\{\phi_t:M\to M\}_{t\in\R}$ and assume its center-stable/unstable leaves $\Fcs(\gamma)$ and $\Fcu(\gamma)$ are cylinders. In \cite{Fri} Fried introduced a surgery operation, represented in Figure \eqref{fig_Fried_surgery}, that consists in:
\begin{enumerate}
\item \emph{Choose a \textbf{blow-up} $(\phi^*,M^*)$ of the flow $(\phi,M)$ along $\gamma$, whose restriction to the boundary $\{\phi_t^*:\partial M^*\to \partial M^*\}_{t\in\R}$ is Morse-Smale.}

\smallskip
\noindent
A \emph{blow-up} of the pair $(\phi,M)$ along $\gamma$ is a 3-manifold $M^*$ with boundary homeomorphic to a torus, equipped with a flow 
$\{\phi_t^*:M^*\to M^*\}_{t\in\R}$ and a map $\pi:M^*\to M$ such that: (i)- $\pi:\partial M^*\to\gamma$ is a $\mathbb{S}^1$-fibration; (ii)- $\pi:M^*\setminus\partial M^*\to M\setminus\gamma$ is a homeomorphism; and (iii)- $\pi\circ\phi_t^*=\phi_t\circ\pi$, $\forall\ t\in\R$. There always exists a blow-up where the flow on the boundary is a \emph{Morse-Smale flow} with two attracting and two repelling periodic orbits.  

\smallskip
\item \emph{Make a non-trivial \textbf{blow-down} of the boundary.} 

\smallskip
\noindent
For this, choose a foliation $\mathcal{C}$ by simple closed curves on $\partial M^*$ such that each leaf is transverse to the flow $\{\phi_t^*:\partial M^*\to\partial M^*\}_{t\in\R}$ and intersects once each periodic orbit in the boundary. Up to a time reparametrization near the boundary we may assume that $\phi^*$ preserves this foliation. Then, the quotient $\sigma:M^*\to N\coloneqq M^*/\mathcal{C}$, that consists in collapsing each leaf of $\mathcal{C}$ into a point, is a $3$-manifold where $\delta\coloneqq \partial M^*/\mathcal{C}$ is a simple closed curve, and there is an induced flow $\{\psi_t:N\to N\}_{t\in\R}$ that has $\delta$ as a saddle type periodic orbit.   
\end{enumerate}  

\begin{figure}[t]
\label{fig_Fried_surgery}
\centering
\includegraphics[width=0.8\textwidth]{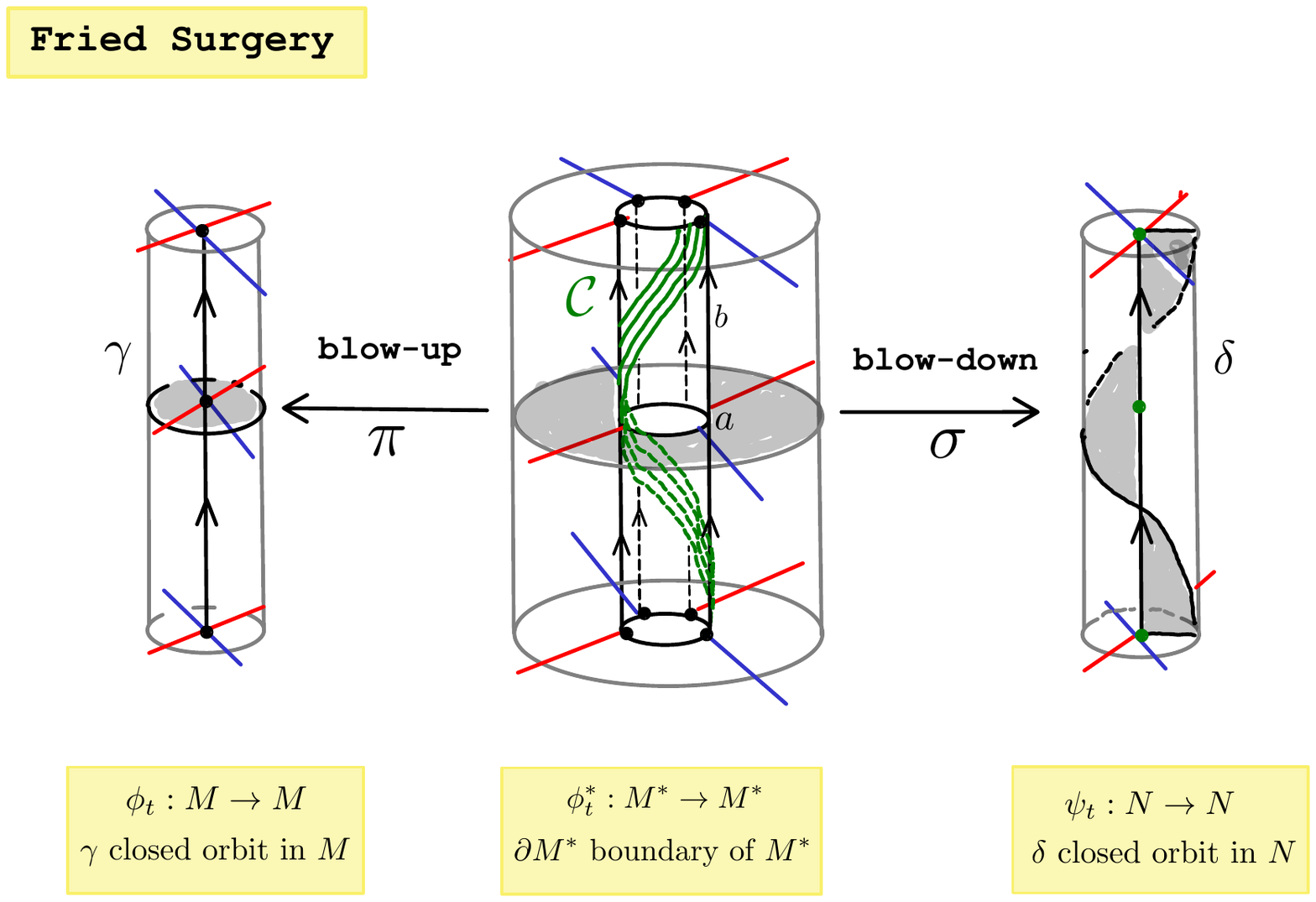}
\caption{Fried surgery $(\phi,M)\xrightarrow[]{(\gamma,1)}(\psi,N)$.}
\end{figure}

Topologically, this operation is an \emph{integral Dehn surgery} on $M$ along the simple closed curve $\gamma$. Thus the homeomorphism type of $N$ is determined by the homology class in $\partial M^*$ of the leaves $C$ of the chosen foliation $\mathcal{C}$. In the basis of $H_1(\partial M^*)$ given by the class $a$ of an oriented $\mathbb{S}^1$-fiber (\emph{meridian}) and the class $b$ of an oriented closed orbit (\emph{longitude}) we have that $[C]=a+m\cdot b$ for some $m\in\Z$. 

This surgery has the additional property that the resulting manifold is equipped with a flow $\{\psi_t:N\to N\}_{t\in\R}$, that is actually equivalent to $(\phi,M\setminus\gamma)$ in the complement of $\delta$. It is possible to check that the orbit $\delta$ is saddle type, and that the flow $\psi$ preserves a pair of globally defined center-stable and center-unstable foliations on $N$, induced from those corresponding to $\{\phi_t:M\to M\}_{t\in\R}$. However, this flow is not an Anosov flow (i.e. does not preserves a hyperbolic splitting) in a canonical way (see \cite{Sh23} for a deeper discussion). It is what is called a \emph{topological Anosov flow}.  

\begin{rem}
It is conjectured that every topological Anosov flow on a closed 3-manifold is orbitally equivalent to a smooth Anosov flow (Bonatti-Wilkinson, \cite{Bonatti-Wilkinson_PH_diffeos}), and in \cite{Sh_top-Anosov=Anosov} this has been proven to be true in the special case where the flow is transitive (i.e. there exists a point whose orbit is dense in the 3-manifold). From now on, we will state everything under this assumption and make no more distinction between Anosov or topological Anosov.
\end{rem}

\begin{defn}
We say that the previous operation is an \emph{integral Fried surgery of slope $m$} on the flow $(\phi,M)$ along the orbit $\gamma$. We use the notations 
\begin{align*}
& (\psi,N)=\surg(\phi,M;(\gamma,m))\\
&(\phi,M) \xrightarrow[]{(\gamma,m)}(\psi,N)
\end{align*}
to indicate that $(\psi,N)$ is obtained from $(\phi,M)$ by doing the specified surgery.
\end{defn}

Observe that, since the foliation by $\phi$-orbits is unchanged in the complement of $\gamma$ under the Fried surgery, every $\phi$-periodic orbit in $M$ produces a unique $\psi$-periodic orbit in $N$ after surgery. We denote by $\Pi:\textrm{Per}(\phi,M)\to\textrm{Per}(\psi,N)$ the natural correspondence that associates periodic orbits before and after surgery.  

\begin{rem}
The previous surgery can be also performed if $\Fcs(\gamma)$, $\Fcu(\gamma)$ are M\"obius bands. However, in the present note we will only work with Anosov flows having orientable center-stable/unstable foliations.
\end{rem}

%%%%%%%%%%%%%%%%%%%%%%%%%%%%%%%%%%%%%%%%%%%%%%%%%%%%%%%%%%%%%%%%%%%%%%%%%%%%%%%%%%%%%%%%%%%%%%%%%%%%%%%%%%%%%%%%%%%%%%%%%%%%%%%%%%%%%%%%%%%%%%%%%%
\subsection{Graph of Anosov flows}
Given an Anosov flow, we can construct many other non-equivalent ones by making Fried surgeries on its periodic orbits. The problem that arises is to understand, given two Anosov flows, in which way one can be obtained by making Fried surgeries on the other. This can be formalized through the study of the topology of the graph described below. 

We say that two flows are \emph{orbitally equivalent} if there exists a homeomorphism between their supporting manifolds that sends the orbit foliation of one flow to the orbit foliation of the other, preserving the orientation of the orbits. 
Using the methods in \cite{Sh_top-Anosov=Anosov}, it is possible to check that the Fried surgery is a well-defined operation within the family of orbital equivalence classes of (topological) Anosov flows on closed 3-manifolds. This allows to define a graph structure as follows:  

\begin{defn}
\label{defn_graph_AF} 
Define the \emph{graph of Anosov flows}\footnote{See \cite{Hoff} for a similar graph construction within the family of closed 3-manifolds.} to be the 1-complex $\mathbf{G}$ having: 
\begin{itemize}
\item \text{Vertices $\mathbf{G}^0$}: Orbital equivalence classes $[(\phi,M)]$ of Anosov flows on 3-manifolds; 
\item \text{Edges} $\mathbf{G}^1$: For each vertex $v$ in $\mathbf{G}^0$ choose a fixed representative $(\phi,M)$ and for each pair $\gamma\in\textrm{Per}(\phi,M)$ and $m\in\Z$ put an oriented edge 
$$[(\phi,M)]=v \xrightarrow[]{(\gamma,m)} w=[(\psi,N)]$$
connecting the class $v$ to the class $w$ of $(\psi,N)=\surg(\phi,M;(\gamma,m))$. 
\end{itemize}
\end{defn} 

Paths in the directed graph $\mathbf{G}$ encode surgery operations within the family of orbital equivalence classes of Anosov flows, and hence the topological complexity of the graph $\mathbf{G}$ reflects somehow the complexity of the surgery operations on Anosov flows. 

\begin{rem}
Other versions of a graph encoding surgery operations between orbital equivalence classes of Anosov flows can be given, notably with smaller number of edges. For instance, we can define another graphs by: only considering edges $(\gamma,\pm 1)$ in $\mathbf{G}^1$ (and removing all the others);
identifying edges $(\gamma_1,m)\sim(\gamma_2,m)$ in $\mathbf{G}^1$ whenever $\gamma_1$ and $\gamma_2$ are equivalent under a self-orbital equivalence of $(\phi,M)$; or by simply adding a non-oriented edge between two classes whenever there exist at least one surgery connecting them. The results presented in this note are adapted for the case of Definition \ref{defn_graph_AF} above.     
\end{rem}

Anosov flows can be transitive or not, their center-stable and center-unstable foliations can be orientable or not, and their supporting 3-manifold can be orientable or not. Since Fried surgeries preserve each of these properties, the graph $\mathbf{G}$ has many connected components. One fundamental open problem is the following:

\begin{problem}[Christy-Fried-Ghys, see \cite{Kirby_problems}- Problem 3.54]
\label{problem_connectivity}
Let $\mathbf{G}_\circ$ be the subgraph of $\mathbf{G}$ whose vertices represent Anosov flows on \emph{orientable closed 3-manifolds}, dynamically \emph{transitive} and having \emph{orientable center-stable and unstable foliations}. Is the graph $\mathbf{G}_\circ$ connected ?
\end{problem}

One way to investigate the topology of $\mathbf{G}_\circ$ (or even $\mathbf{G}$) consists in fixing one Anosov flow, fixing a set of $n>0$ periodic orbits, and looking at all possible surgeries along the chosen curves. This approach is studied by Bonatti and Iakovoglou in \cite{Bonatti-Iakovoglou_surgeries}, for the case where the Anosov flow is a \emph{suspension} (cf. Section \ref{sec_Anosov_sol-mflds}) and $n=2$. 

If $\gamma_1$, $\gamma_2$ are periodic orbits of a transitive Anosov flow $(\phi,M)$ and $m_1$, $m_2$ are integers, for each $i=1,2$ let $(\phi_i,M_i)=\surg(\phi,M;(\gamma_i,m_i))$ be an Anosov flow obtained by the corresponding Fried surgery, and let $\Pi_i:\textrm{Per}(\phi,M)\to\textrm{Per}(\phi_i,M_i)$ be the map that identifies periodic orbits before and after surgery, as defined before. The methods at Section 3 of \cite{Sh_top-Anosov=Anosov} allow to prove that 
$$\surg(\phi_1,M_1;(\Pi_1(\gamma_2),m_2))\ \text{is orbitally equivalent to}\ \surg(\phi_2,M_2;(\Pi_2(\gamma_1),m_1)),$$
and hence $\surg(\phi^A,M_A;(\gamma_1,m_1),(\gamma_2,m_2))$ is a well-defined class $[(\psi,N)]$ of Anosov flow, independent of the order in which the surgeries along $\gamma_1$ and $\gamma_2$ are performed. 

Surgeries on a pair of periodic orbits yield a diagram:  
\begin{center}
\begin{tikzpicture}[
    node distance=1.5cm,
%    every node/.style={rectangle, draw, inner sep=1.5pt},
    >={Stealth}
]

\node (bot) at (0,0) {$\left[(\phi_2,M_2)\right]$};
\node (a) at (-4,1) {$\left[(\phi,M)\right]$};
\node (b) at (4,1) {$\left[(\psi,N)\right]$};
\node (top) at (0,2) {$\left[(\phi_1,M_1)\right]$};

\draw[->] (a) -- node[below] {$(\gamma_2,m_2)$}(bot);
\draw[->] (a) -- node[above] {$(\gamma_1,m_1)$}(top);
\draw[->] (bot) -- node[below] {$(\Pi_2(\gamma_1),m_1)$}(b);
\draw[->] (top) -- node[above] {$(\Pi_1(\gamma_2),m_2)$}(b);
\end{tikzpicture}
\end{center}
We will use the notation 
$$[(\phi,M)]\xrightarrow[]{(\gamma_1,m_1),(\gamma_2,m_2)}[(\psi,N)]$$  
to indicate that $(\psi,N)$ is obtained, up to orbital equivalence, by doing surgeries of slopes $m_1$ and $m_2$ along the periodic orbits $\gamma_1$ and $\gamma_2$ of $(\phi,M)$, respectively. This represents the existence of two paths in $\mathbf{G}_\circ$ connecting the classes $[(\phi,M)]$ and $[(\psi,N)]$, each one having length $2$, as indicated in the diagram above.

The work in \cite{Bonatti-Iakovoglou_surgeries} shows that:

\begin{thm}[Bonatti-Iakovoglou]
Let $(\phi^A,M_A)$ be the suspension flow generated by a hyperbolic matrix $A\in\SL(2,\Z)$ with $\tr(A)\geq 3$. Then, 
\begin{enumerate}
\item For every pair $(\gamma_1,\gamma_2)$ of different periodic orbits, there is at most finitely many pairs $(m_1,m_2)$ of integers such that 
$$\surg(\phi^A,M_A;(\gamma_1,m_1),(\gamma_2,m_2))$$
is equivalent to a suspension Anosov flow (cf. \cite[Proposition 9.1]{Bonatti-Iakovoglou_surgeries}).

\item There exist special pairs $(\gamma_1,\gamma_2)$ of periodic orbits where every surgery 
$$\surg(\phi^A,M_A;(\gamma_1,m_1),(\gamma_2,m_2))$$
with $(m_1,m_2)\neq (0,0)$ produces an Anosov flow that is not equivalent to a suspension (cf. \cite[Theorem 5]{Bonatti-Iakovoglou_surgeries}).
\end{enumerate} 
\end{thm}

\begin{rem}\label{rem_Bonatti-Iako}
Along the discussion at \cite[Section 1.3]{Bonatti-Iakovoglou_surgeries} the authors claim that, for every fixed hyperbolic matrix $A\in\SL(2,\Z)$, it seems plausible the existence of \emph{at most finitely many} 4-tuples $(\gamma_1,m_1,\gamma_2,m_2)$ with $\gamma_1\neq\gamma_2$ and $(m_1,m_2)\neq(0,0)$ such that the surgery $\surg(\phi^A,M_A;(\gamma_1,m_1),(\gamma_2,m_2))$ produces a suspension Anosov flow. 
\end{rem}

In this note we describe a simple construction that, for any given suspension Anosov flow, allows to obtain infinitely many pairs of periodic orbits $\gamma_1$, $\gamma_2$ and integers $m$ where the surgery with data $\{(\gamma_1,m),(\gamma_2,-m)\}$ yields another suspension Anosov flow. In particular, this shows that the expected claim at Remark \ref{rem_Bonatti-Iako} does not hold. Moreover, when the generating matrix $A$ of the suspension flow is conjugated to its inverse, this construction provides infinitely many non-trivial loops of length 2 in $\mathbf{G}_\circ$ with base point at $[(\phi^A,M_A)]$.

%%%%%%%%%%%%%%%%%%%%%%%%%%%%%%%%%%%%%%%%%%%%%%%%%%%%%%%%%%%%%%%%%%%%%%%%%%%%%%%%%%%%%%%%%%%%%%%%%%%%%%%%%%%%%%%%%%%%%%%%%%%%%%%%%%%%%%%%%%%%%%%%%%
\subsection{Infinitely many loops of length 2 based at a suspension}
The main result of this note is the following statement:

\begin{thmA}
\label{thm_A}
Let $A\in\SL(2,\Z)$ be a matrix with $\tr(A)\geq 3$ satisfying that $A$ is conjugated to $A^{-1}$ in $\GL(2,\Z)$. Then, given any periodic orbit $\gamma$ of the suspension flow $(\phi^A,M_A)$ there exists $m_0=m_0(\gamma)>0$ such that, for every integer $m$ with $|m|\geq m_0$, there is a set $\mathcal{Q}(\gamma,m)$ containing infinitely many $\phi^A$-periodic orbits, satisfying that 
$$[(\phi^A,M_A)]\xrightarrow[]{(\gamma,m),(\alpha,-m)}[(\phi^A,M_A)],\ \text{for every}\ \alpha\in\mathcal{Q}(m,\gamma).$$
Moreover, let $\Pi:\textrm{Per}(\phi^A, M_A)\to\textrm{Per}(\phi^A, M_A)$ be the correspondence that identifies periodic orbits before and after this surgery operation. Then, the image of $\alpha$ under the correspondence $\Pi$ equals the periodic orbit $\gamma$, up to a self-orbital equivalence of $(\phi^A,M_A)$.
\end{thmA}

In other words, for hyperbolic matrices $A$ conjugated to its inverse, we find infinitely many different loops of length $2$ in $\mathbf{G}_\circ$ based at the class $[(\phi^A,M_A)]$. Each of these loops consists in making one surgery of slope $m$ along one periodic orbit $\gamma$, together with one surgery of slope $-m$ along a different periodic orbit $\alpha$ chosen from the set $\mathcal{Q}(m,\gamma)$, and $m\in\Z$ may be taken arbitrarily big. Since $\mathcal{Q}(m,\gamma)$ is infinite, it contains orbits $\alpha$ which cannot be sent into $\gamma$ using a self-orbit equivalence of $(\phi^A,M_A)$. The fact that $\Pi(\alpha)=\gamma$ up to self-orbital equivalence means that the way in which the periodic orbits of $(\phi^A,M_A)$ identify with themselves under this double-surgery operation is not trivial, in the sense that $\Pi$ is not induced from a self-orbital equivalence of $(\phi^A,M_A)$. This means that the surgery encoded on this loop is not trivial either. 

It is important to remark that, for every fixed $m$ in $\Z$ and $\gamma$ in $\PP$, the set of loops described in Theorem \textbf{A} is small, in the following sense: 

\begin{propB}
Let $\gamma$ be a periodic orbit of $(\phi^A,M_A)$, let $m$ be an integer satisfying that $|m|\geq m_0(\gamma)$, and let $\mathcal{Q}(\gamma,m)$ be the set of periodic orbits described in Theorem \textbf{A}. Then:
$$\lim_{t\to\infty}\frac{\left\vert\left\{\alpha\in\mathcal{Q}(\gamma,m):\per(\alpha)<t\right\}\right\vert}
{\left\vert\left\{\alpha\in\PP):\per(\alpha)<t\right\}\right\vert}=0.$$
\end{propB}

Observe that the previous asymptotic is given in terms of a fixed $m$ and $\gamma$. It would be interesting to find an asymptotic on the total number of loops
$$[(\phi^A,M_A)]\xrightarrow[]{(\gamma,m),(\alpha,-m)}[(\phi^A,M_A)],\ \alpha\in\mathcal{Q}(\gamma,m)$$ 
with $m$ and $\gamma$ varying among all possible choices. 

\begin{question}
Is it possible to estimate the growth rate, as a function of the period, of the \emph{total} number of loops in $\mathbf{G}_\circ$ of length 2 (or any other fixed length) based at $[(\phi^A,M_A)]$ ? 
\end{question}

\begin{rem}
As a final remark, when the matrix $A$ is not conjugated to $A^{-1}$ in $\GL(2,\Z)$, the construction of Theorem \textbf{A} provides paths of length $2$ connecting two suspensions
$$[(\phi^A,M_A)]\xrightarrow[]{(\gamma,m),(\alpha,-m)}[(\phi^B,M_B)],\ \forall\ \alpha\in\mathcal{Q}(m,\gamma),$$ 
where $B$ is conjugated to either $A$ or $A^{-1}$. However, determining the exact conjugacy class of $B$ in this case is more difficult.
\end{rem}

%%%%%%%%%%%%%%%%%%%%%%%%%%%%%%%%%%%%%%%%%%%%%%%%%%%%%%%%%%%%%%%%%%%%%%%%%%%%%%%%%%%%%%%%%%%%%%%%%%%%%%%%%%%%%%%%%%%%%%%%%%%%%%%%%%%%%%%%%%%%%%%%%%%
\section{$\infty$-many closed paths}
The construction of the sets $\mathcal{Q}(\gamma,m)$ in Theorem \textbf{A} is an elementary application of three results due to Plante, Thurston and Barthelm\'e-Fenley. We will summarize these results first and explain how to construct the paths stated in Theorem \textbf{A} at the end.

%%%%%%%%%%%%%%%%%%%%%%%%%%%%%%%%%%%%%%%%%%%%%%%%%%%%%%%%%%%%%%%%%%%%%%%%%%%%%%%%%%%%%%%%%%%%%%%%%%%%%%%%%%%%%%%%%%%%%%%%%%%%%%%%%%%%%%%%%%%%%%%%%%%
\subsection{Anosov flows on \emph{sol}-manifolds}
\label{sec_Anosov_sol-mflds}
A \emph{suspension Anosov flow} is one constructed by choosing a hyperbolic matrix $A\in\GL(2,\Z)$ and making the quotient of the space 
$\T^2\times\R$ by the action $(x,t)\mapsto(A(x),t-1)$, where $\T^2=\R^2/\Z^2$. This gives a closed $3$-manifold $M_A$ equipped with an Anosov flow 
$\{\phi_t^A:M_A\to M_A\}_{t\in\R}$ induced from the action $\tilde{\phi}_t(x,s)=(x,s+t)$ on $\T^2\times\R$. The stable and unstable invariant bundles of the flow are induced from the contracting and expanding eigenspaces of the matrix $A$ acting on $\T^2$.  

The quotient $\T^2\times\R\to M_A$ induces a \emph{torus bundle} $\pi_+:M_A\to\R/\Z$ with monodromy equal to $A$. For each $\theta\in\R/\Z$, the torus $\pi_+^{-1}(\theta)$ is a global transverse section of the suspension flow, with first return map realizing the monodromy of the bundle. An Anosov flow is orbitally equivalent to a suspension if and only if it has a global transverse section, necessarily homeomorphic to a torus.

A closed $3$-manifold is a \emph{sol-manifold} if its fundamental group is \emph{solvable}. This is the case of the supporting 3-manifold of a suspension Anosov flow, whose fundamental group is isomorphic to an affine subgroup of the Euclidean space and hence solvable. There is a classification up to orbital equivalence of Anosov flows on sol-manifolds due to Plante:

\begin{thm}[Plante, \cite{Plante_Anosov_sol-mfls}. See also \cite{Barbot-Maquera_cod=1_Anosov_actions}]
\label{thm_classific_sol-mflds}
Let $\{\phi_t:M\to M\}_{t\in\R}$ be an Anosov flow on a closed 3-manifold with solvable fundamental group. Then, $(\phi,M)$ is orbitally equivalent to a suspension Anosov flow $(\phi^A,M_A)$, where $A\in\GL(2,\Z)$ is a hyperbolic matrix.  
\end{thm}

On the 3-manifold $M_A$ there is another Anosov flow $\{\bar{\phi}_t^A:M_A\to M_A\}_{t\in\R}$ and another torus bundle $\pi_-:M_A\to\R/\Z$ with monodromy $A^{-1}$, such that its fibers $\pi_-^{-1}(\theta)$ are global transverse sections for $\bar{\phi}^A$ with first return conjugated to the monodromy. These are obtained by reversing the time direction of $\phi^A$ and reversing the orientation of the base in the bundle projection $\pi_+$. Both flows $\phi^A$ and $\bar{\phi}^A$ define the same foliation by orbits of $M_A$, up to changing the orientation of the orbits, and both maps $\pi_+$ and $\pi_-$ define the same fibration by tori of $M_A$, up to changing the coorientation of the fibers. There is a rigidity property of these foliations that allows to sharpen the statement in Theorem \ref{thm_classific_sol-mflds} above: 

\begin{thm}[Uniqueness of the $\T^2$-fibration]
\label{prop_Uniqueness_of_torus_fibration}
Let $A\in\GL(2,\Z)$ be a hyperbolic matrix and let $M_A$ be the 3-manifold obtained by suspension of $A$. If $\pi:M_A\to\s$ is a torus bundle, then there exists a homeomorphism $H:M_A\to M_A$ isotopic to the identity, such that $\pi=\pi_\varepsilon\circ H$, where either $\varepsilon=+$ or $\varepsilon=-$.
\end{thm}

This \emph{folkloric} statement follows by direct application of Thurston's \emph{fibered faces theory}, see e.g. \cite{Thurston_norm}. Since $H^1(M_A,\Z)$ has rank equal to $1$ then there is a unique fibration of $M_A$ by closed surfaces, up to isotopy. As a consequence one has that:

\begin{prop}
\label{prop_rigidity}
Let $A\in\GL(2,\Z)$ be a hyperbolic matrix and let $M_A$ be the 3-manifold obtained by suspension of $A$. Then every Anosov flow $\{\psi_t:M_A\to M_A\}_{t\in\R}$ is orbitally equivalent to either $\phi^A$ or $\bar{\phi}^A$, under an orbital equivalence isotopic to the identity on $M_A$.
\end{prop}

\begin{proof}
If $\{\psi_t:M_A\to M_A\}_{t\in\R}$ is any Anosov flow, since $\pi_1(M_A,x_0)$ is solvable then by Theorem \ref{thm_classific_sol-mflds} it is equivalent to a suspension flow $(\phi^B,M_B)$ with $B\in\GL(2,\Z)$ hyperbolic. Hence, there is a fibration $\pi:M_A\to\s$ by tori $\pi^{-1}(\theta)$, $\theta\in\R/\Z$ each of them being a global transverse section for $\psi$ with monodromy conjugated to $B$. By Theorem \ref{prop_Uniqueness_of_torus_fibration} there is a homeomorphism $H:M_A\to M_A$ isotopic to the identity such that $\pi=\pi_\varepsilon\circ H$, where $\varepsilon=+,-$. Hence, the monodromy $B$ of the bundle $\pi$ is conjugated either $A$ of $A^{-1}$, and the homeomorphism $H$ realizes an orbital equivalence between $(\psi,M_A)$ and either $(\phi^A,M_A)$ or $(\bar{\phi}^A,M_A)$. \qedhere
\end{proof}

Finally, we will use a property of the foliation by orbits of $(\phi^A,M_A)$, stated in Proposition \ref{prop_rigidity2} below. Let's introduce some terminology.  

If $\gamma$ is an oriented simple closed curve in a $3$-manifold $M$ and $W_\gamma\simeq \D^2\times\s$ is a tubular neighborhood, it is possible to choose an oriented \emph{meridian-longitude} basis $\{a_\gamma,b_\gamma\}$ of $H_1(W_\gamma\setminus\gamma)\simeq H_1(\partial W_\gamma)$, by setting $a_\gamma\in H_1(\partial W)$ to be homology class of the boundary of a meridian disk $a_\gamma=[\partial\D^2\times\{\theta\}]$, and $b_\gamma\in H_1(\partial W)$ to be the class of any simple closed curve having oriented intersection number $a_\gamma\cdot b_\gamma=+1$. For every closed curve $\sigma$ in $\partial W_\gamma$ define its \emph{multiplicity} $m\in\Z$ to be the intersection number $m=a_\gamma\cdot[\sigma]$. Given two oriented simple closed curves $\gamma$ and $\delta$ respectively contained in $M$ and $N$, we say that a homeomorphism $h:M\setminus\gamma\to N\setminus\delta$ \emph{preserves the orientation at infinity} if when restricted to tubular neighborhoods $h:W_\gamma\setminus\gamma\to W_\delta\setminus\delta$ he have that $a_\delta\cdot h_*(b_\gamma)=+1$. Observe that this is a necessary condition for $h$ to extend onto a homeomorphism $M\to N$ that takes $\gamma\mapsto\delta$ preserving their orientations. 

\begin{prop}
\label{prop_rigidity2}
Let $(\phi^A,M_A)$ be the suspension Anosov flow generated by a hyperbolic matrix $A\in\GL(2,\Z)$, and let $\gamma$ and $\delta$ be two periodic orbits. If there exists a homeomorphism $M_A\setminus\gamma\to M_A\setminus\delta$ preserving the orientations at infinity of $\gamma$ and $\delta$, then there exists a self-orbital equivalence $\bar{H}:(\phi^A,M_A)\to(\phi^A,M_A)$ such that $\bar{H}(\gamma)=\delta$.
\end{prop}

\begin{proof}
To show the proposition it is enough to show the existence of a homeomorphism $H:M_A\to M_A$ such that $H(\gamma)=\delta$, preserving the orientations of these curves. To see this, let $\psi$ be the flow on $M_A$ obtained by pushing forward the orbits of $\phi^A$ with $H$. Then, by Proposition \ref{prop_rigidity} there is an orbital equivalence
\footnote{Note that $\psi$ is technically a topological Anosov flow. However, Proposition \ref{prop_rigidity} is still valid due to the results in \cite{Sh_top-Anosov=Anosov}. See also \cite{Bru_expansive_sol-mflds}.}
$H^*:(\psi,M_A)\to(\phi^A,M_A)$ isotopic to the identity. Since $H^*(\delta)$ is isotopic to $\delta$ then $H^*(\delta)=\delta$, because for every suspension Anosov flow there exist at most one periodic orbit per (free) isotopy class (see e.g. \cite{Barthelme-Fenley_II}). Hence, the homeomorphism $\bar{H}=H^*\circ H$ is a self-orbital equivalence of the suspension flow taking $\bar{H}:\gamma\mapsto\delta$.

To construct the homeomorphism $H$ we will use the following: 

\begin{cl}
Let $h_0:W\setminus\alpha\to W\setminus\alpha$ be a homeomorphism. Then there exists a homeomorphism $H_0:W\to W$ that coincides with $h_0$ on $\partial W$ if and only if the curves $h_0(\partial\D^2\times\{\theta\})$ have multiplicity $m=0$.
\end{cl}

\begin{proof}[Proof of claim]
To prove this, observe that the existence of $H_0$ implies that for every $\theta\in\R/\Z$ then $H_0(\D^2\times\{\theta\})$ is a meridian disks of $W$, and hence $h_0(\partial\D^2\times\{\theta\})=\partial H_0(\D^2\times\{\theta\})$ has multiplicity equal to zero. Reciprocally, take one curve $h_0(\partial\D^2\times\{\theta_0\})$. Since it has zero multiplicity, this curve is the boundary of an embedded meridian disk $D_0\to W$. This allows to construct a homeomorphism $\partial W\cup \D^2\times\{\theta_0\}\to\partial W\cup D_0$ that coincides with $h_0$ on $\partial W$. By splitting the solid torus $W$ along $\D^2\times\{\theta_0\}$ and $D_0$ we obtain two 3-balls with a homeomorphism defined between their boundaries. An extension of this homeomorphism to the interior of the balls provides the homeomorphism $H_0$. \qedhere
\end{proof}

Let $h:M_A\setminus\gamma\to M_A\setminus\delta$ be a homeomorphism. Let $W_\gamma$ and $W_\delta$ be small closed tubular neighborhoods around $\gamma$ and $\delta$ respectively and assume without loss of generality that $W_\delta\setminus\delta=h(W_\gamma\setminus\gamma)$. Let $\Sigma=\pi_+^{-1}(\theta)\cap(M_A\setminus\mathring{W}_\gamma)$ be the intersection of a torus fiber with the complement of the tubular neighborhood. We have that $\Sigma$ is a closed surface with $p>0$ boundary components $C_1,\dots,C_p$ in $\partial W_\gamma$ isotopic to meridians. Let $m$ be the multiplicity of $h(C_i)\subset\partial W_\delta$. Since the homomorphism $H_1(\partial W_\delta)\to H_1(W_\delta)$ induced by the inclusion has kernel generated by the meridian $a$, we can see that $h(\partial\Sigma)$ is homologous to $p\cdot m\cdot [\delta]$ in $H_1(W_\delta)$. Since the homomorphism $H_1(W_\delta)\to H_1(M_A)$ induced by inclusion is injective, we have that 
$$0=[h(\partial\Sigma)]=p\cdot m\cdot [\delta]\in H_1(M_A).$$
Since $[\delta]\neq 0$ for every periodic orbit $\delta$ of the suspension flow, we conclude that $m=0$. Therefore, using the claim above we can  modify the homeomorphism $h$ inside the tubular neighborhoods $W_\gamma$ and $W_\delta$, and obtain a globally defined $H:M_A\to M_A$ as required. \qedhere
\end{proof}

%%%%%%%%%%%%%%%%%%%%%%%%%%%%%%%%%%%%%%%%%%%%%%%%%%%%%%%%%%%%%%%%%%%%%%%%%%%%%%%%%%%%%%%%%%%%%%%%%%%%%%%%%%%%%%%%%%%%%%%%%%%%%%%%%%%%%%%%%%%%%%%%%%%
\subsection{Hyperbolic Dehn surgery}
Let $M$ be a closed 3-manifold, $\gamma$ a simple closed curve, and assume that $M\setminus\gamma$ admits a complete hyperbolic metric. A classical theorem of Thurston (see \cite{Thurston_3-mfld_topology}, \cite{ratcliffe}) states that making a Dehn surgery on $M$ along $\gamma$ will produce a hyperbolic closed 3-manifold, for all except a finite number of surgery slopes. This method for constructing hyperbolic 3-manifolds is known as \emph{hyperbolic Dehn surgery}. 

Hyperbolic Dehn surgery can be implemented in the context of Anosov flows and Fried surgeries to produce Anosov flows on hyperbolic 3-manifolds. This method goes back to Goodman\footnote{In Goodman's work \cite{Goodman_surgery} other version of the Dehn surgery is used. See \cite{Sh_thesis} for an explanation of how the Goodman and Fried surgeries are related.} in \cite{Goodman_surgery}, where the first examples of Anosov flows on hyperbolic 3-manifolds were obtained. These examples can be described as follows: take the suspension flow $(\phi^A,M_A)$ associated to the matrix $A=[2,1,1,1]$ and make a Fried surgery of slope $m\neq 0$ along the periodic orbit $\gamma_0$ associated to the (unique) fixed point of $A$. It is a well-known fact (cf. Chapter 8 of \cite{francis}) that the complement of $\gamma_0$ in $M_A$ is homeomorphic to the complement of the \emph{figure-eight knot} in the sphere $\mathbb{S}^3$, which admits a complete hyperbolic metric (cf. \cite{Thurston_3-mfld_topology}). Hence, when the slope $m$ of the surgery is chosen big enough, we obtain an Anosov flow on a hyperbolic 3-manifold. We refer to \cite{binyu_figure-eight} for a treatment of the Anosov flows obtained in this particular example and its relation to the figure-eight knot complement. 

The method described above is completely general and does not depend on the special matrix $A=[2,1,1,1]$. It can be shown that, for every $A\in\SL(2,\Z)$ with $\tr(A)\geq 3$, the complement of any periodic orbit $\gamma$ of the suspension Anosov flow $(\phi^A,M_A)$ admits a complete hyperbolic metric. This follows from Thurston's hyperbolization theorem at \cite{thurston_hyperbolization_II}. Since $M_A\setminus\gamma$ fibers over the circle with fibers homeomorphic a punctured torus, and since the monodromy $A$ on the fibers is of pseudo-Anosov type, then case (iii) of Theorem 0.1 at \cite{thurston_hyperbolization_II} applies and $M_A\setminus\gamma$ is hyperbolic. We refer to \cite{shannon_pinsky_fenley} for a general treatment of the problem of whether the complement of a periodic orbit of an Anosov flow admits a complete hyperbolic metric. See also \cite{gueritaud_futer}, \cite{hatcher_1-punctured_torus}.

Putting all together, we get the following statement:

\begin{prop}
\label{prop_hyperbolic_Fried_surgery}
Consider the suspension flow $(\phi^A,M_A)$ generated by $A\in\SL(2,\Z)$ with $\tr(A)\geq 3$. Then, for every $\phi^A$-periodic orbit $\gamma$ the Fried surgery along $\gamma$ with slope $m$ will produce an Anosov flow on a hyperbolic 3-manifold, for all except a finite number of $m\in\Z$.
\end{prop}

%%%%%%%%%%%%%%%%%%%%%%%%%%%%%%%%%%%%%%%%%%%%%%%%%%%%%%%%%%%%%%%%%%%%%%%%%%%%%%%%%%%%%%%%%%%%%%%%%%%%%%%%%%%%%%%%%%%%%%%%%%%%%%%%%%%%%%%%%%%%%%%%%%%
\subsection{$\R$-covered Anosov flows on hyperbolic 3-manifolds}
Let $\{\phi_t:M\to M\}_{t\in\R}$ be an Anosov flow on a closed 3-manifold and denote by $\mathcal{O}$ the corresponding foliation of $M$ by $\phi$-orbits. Let $\tilde{M}$ be the universal cover of $M$ and let $\tilde{\mathcal{O}}$ be the lift of the foliation $\mathcal{O}$. The following is a fundamental result due to Barbot and Fenley:

\begin{figure}[t]
\centering
\includegraphics[width=\textwidth]{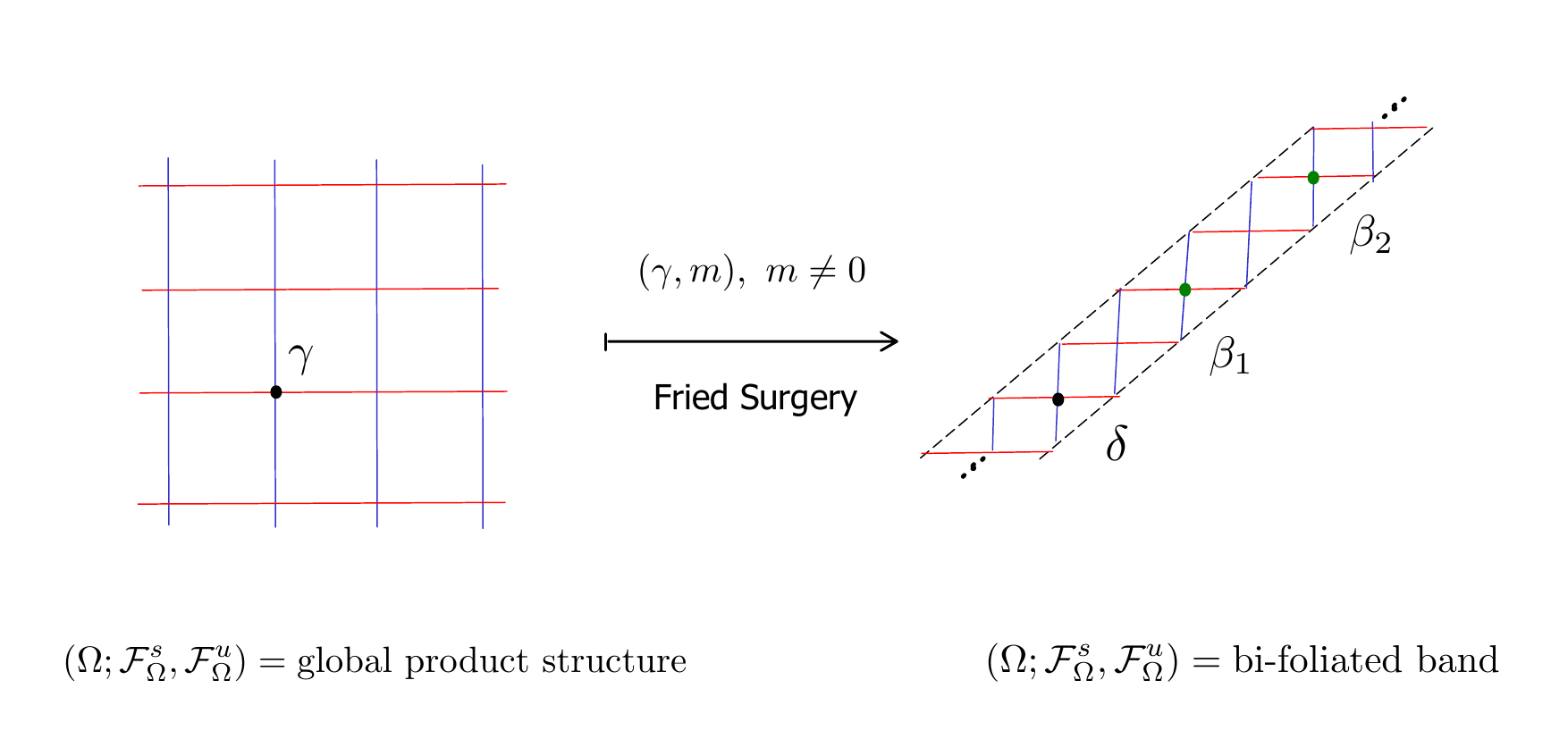}
\caption{$\R$-covered Anosov flows}
\label{fig_Fried_surgery(2)}
\end{figure}

\begin{thm}[Barbot, \cite{Barbot_characterisation_feuilletages_faibles}- Fenley, \cite{Fenley_Anosov-flows}]
\label{thm_Fenley-Barbot}
There exists a homeomorphism $\tilde{M}\to\R^3$ taking $\tilde{\mathcal{O}}$ to the foliation by \emph{vertical} lines $\mathcal{V}=\{z\times\R:z\in\R^2\}$ of $\ \R^3=\R^2\times\R$.   
\end{thm}

As a consequence, the quotient $\Omega\coloneqq\tilde{M}/\tilde{\mathcal{O}}$ is homeomorphic to $\R^2$. The center-stable and center-unstable foliations of $M$ lift to a pair of codimension 1 transverse foliations of $\tilde{M}$ that intersect along $\tilde{\mathcal{O}}$-leaves, and hence induce a pair of 1-dimensional transverse foliations of $\Omega$. We denote them by $\Fs_\Omega$ and $\Fu_\Omega$. 

The triple $\left(\Omega;\Fs_\Omega,\Fu_\Omega\right)$ is called the \emph{bi-foliated plane} associated to $(\phi,M)$. A general bi-foliated plane may have several different natures. A case of particular interest is given by the following result:

\begin{thm}[\cite{Barbot_characterisation_feuilletages_faibles}, \cite{Fenley_Anosov-flows}]
\label{thm_R-covered}
Consider the quotient spaces $\Lambda^s\coloneqq\Omega/\Fs_{\Omega}$ and $\Lambda^u\coloneqq\Omega/\Fu_{\Omega}$. Then $\Lambda^s$ is homeomorphic to the real line if and only if $\Lambda^u$ is homeomorphic to the real line, and in this case there are two possibilities, represented in Figure \eqref{fig_Fried_surgery(2)}:

\begin{enumerate}[(i)]
\item Either $(\Omega;\Fs_\Omega,\Fu_\Omega)$ is equivalent to $\R^2$ endowed with the foliations by \emph{horizontal} and \emph{vertical} lines. In this case $(\phi,M)$ is orbitally equivalent to the suspension flow generated by a hyperbolic matrix $A\in\SL(2,\Z)$.

\item Or $(\Omega;\Fs_\Omega,\Fu_\Omega)$ is equivalent to the space obtained by restricting the \emph{horizontal} and \emph{vertical} foliations of $\R^2$ onto an open band delimited by two parallel lines, none of angle $\theta= 0$ or $\theta=\pi/2$.  
\end{enumerate}
\end{thm}  

When $\Lambda^s\simeq\Lambda^u\simeq\R$ the flow $(\phi,M)$ is said to be $\R$-\emph{covered}. In case (i) the bi-foliated plane is said to have \emph{global product structure}, while in case (ii) is said to be \emph{skewed $\R$-covered}, and it can be \emph{positively} or \emph{negatively twisted} according to the sign of $-\pi/2<\theta<\pi/2$. Case (i) is rigid, in the sense that there is only one family of Anosov flows admitting this bi-foliated plane structure, namely, the suspension Anosov flows. Case (ii) is more flexible, there are several families of Anosov flows having this kind of bi-foliated plane. For instance, \emph{geodesic flows on hyperbolic surfaces} and \emph{contact Anosov flow} are in this category (see e.g. \cite{Barbot_plane_aff_geom}). We will use the following result: 

\begin{prop}[See \cite{Fenley_Anosov-flows} and \cite{Bonatti-Asaoka-Marty_R-cov}]
\label{prop_R-cov_Fried_surgery}
Given a suspension Anosov flow, every Fried surgery of slope $m\neq 0$ along one periodic orbit produces a skewed $\R$-covered Anosov flow, positively or negatively twisted according to the sign of $m\in\Z$. 
\end{prop} 

This results follows directly from Theorem A at \cite{Bonatti-Asaoka-Marty_R-cov}, since a Fried surgery supported in one orbit of a suspension Anosov flow will produce an Anosov flow endowed with a \emph{positive Birkhoff section}, as defined therein. 

There is an interesting phenomenon that happens for skewed $\R$-covered Anosov flows when the supporting 3-manifold is hyperbolic, and is the hearth of our construction:

\begin{thm}[Barthelm\'e-Fenley, \cite{Barthelme-Fenley_I}]
\label{thm_free_homot_classes}
If $\{\psi_t:N\to N\}_{t\in\R}$ is a $\R$-covered Anosov flow on a hyperbolic 3-manifold, then for every periodic orbit $\delta$ there exists an infinite set $\mathcal{FH}(\delta)$ of periodic orbits which are equivalent to $\delta$ as knots. That is, for every $\beta\in\mathcal{FH}(\gamma)$ there exists a 1-parameter family of homeomorphisms $H_t:M\to M$ such that $H_0=\textrm{id}_M$ and $H_1(\beta)=\delta$ preserving orientation of these curves.  
\end{thm}

%%%%%%%%%%%%%%%%%%%%%%%%%%%%%%%%%%%%%%%%%%%%%%%%%%%%%%%%%%%%%%%%%%%%%%%%%%%%%%%%%%%%%%%%%%%%%%%%%%%%%%%%%%%%%%%%%%%%%%%%%%%%%%%%%%%%%%%%%%%%%%%%%%%
\subsection{Birkhoff sections}
Theorem \textbf{A} can be rephrased in terms of Birkhoff sections. Given an Anosov flow $\{\phi_t:M\to M\}_{t\in\R}$ a \emph{Birkhoff section} is an immersion of a compact surface $\Sigma$ inside $M$, denoted $\iota:(\Sigma,\partial\Sigma)\to(M,\Gamma)$, satisfying that:

\begin{enumerate}
\item On the boundary $\partial\Sigma$ the map $\iota$ is a covering onto a finite union of $\phi$-periodic orbits $\Gamma=\{\gamma_1,\dots,\gamma_k\}$;

\item On the interior $\mathring{\Sigma}\coloneqq\Sigma\setminus\partial\Sigma$ the map $\iota$ is an embedding and $\iota(\mathring{\Sigma})$ is transverse to the $\phi$-orbits;

\item There exists $T>0$ such that $\phi_{[0,T]}(x)\cap\iota(\Sigma)\neq\emptyset$, $\forall\ x\in M$.   
\end{enumerate}

Observe that, since $\mathring{\Sigma}$ is transverse to the $\phi$-orbits, there is always a canonical choice of orientation on $\Sigma$, defined by the convention that the orientation of $\Sigma$ paired with the coorientation determined by the flow is equivalent to the opposite orientation of $M$. For each $\gamma_i\in\Gamma$ we will set:

\begin{itemize}
\item $p_i=p(\gamma_i,\Sigma)\geq 1$ to be the \emph{number of connected components} of $\iota^{-1}(\gamma_i)$, 

\item $m_i=m(\gamma_i,\Sigma)\in\Z$ to be the degree of the oriented covering $\iota:C_{ij}\to\gamma_i$, where $C_{i1},\dots,C_{ip_i}$ are the components of $\iota^{-1}(\gamma_i)$ oriented as boundary of $\Sigma$, and $\gamma_i$ is oriented by the flow-direction. This integer is called the \emph{multiplicity} of the section $\Sigma$ at $\gamma_i$ and is the same for every component (see e.g. \cite{Sh_top-Anosov=Anosov}).  
\end{itemize}

Properties (1), (2), (3) imply that there exists a \emph{first return map} $P_\Sigma:\mathring{\Sigma}\to\mathring{\Sigma}$ on the interior of the section, and hence the restriction of $\phi$ to $M\setminus\Gamma$ is a suspension flow (in an open manifold). There is a construction called \emph{blow-down} which consists collapsing each component of $\partial\Sigma$ into a point. This produces a closed surface $\hat{\Sigma}$ together with a homeomorphism $\hat{P}_\Sigma:\hat{\Sigma}\to\hat{\Sigma}$ induced from $P_\Sigma$. For each $\gamma_i\in\Gamma$, the set of components $\iota^{-1}(\gamma_i)$ corresponds to a periodic orbit of $\hat{P}_\Sigma$ with period $p_i$. The feature of this construction is that $(\phi,M)$ can be obtained from the suspension flow $(\hat{\phi},\hat{M})$ generated by the blow-down homeomorphism (often a \emph{pseudo-Anosov flow}) by making Fried surgeries on the curves $\gamma_1,\dots,\gamma_k$. See \cite{Fri} for more details.

If $\Sigma$ has genus $1$, since $\{\phi_t:M\to M\}_{t\in\R}$ is Anosov and $\mathring{\Sigma}$ is transverse to the $\phi$-orbits, then the center-stable and unstable foliations of the flow print on $\mathring{\Sigma}$ a pair of transverse foliations $\Fs_\Sigma$, $\Fu_\Sigma$, that extend onto foliations on $\hat{\Sigma}$ without singularities (since $\chi(\hat{\Sigma})=0$), that are respectively contracted and expanded by the action of $\hat{P}_\Sigma$. Thus, the homeomorphism $\hat{P}:\hat{\Sigma}\to\hat{\Sigma}$ is conjugated to a linear Anosov diffeomorphism on the torus given by a matrix $A\in\SL(2,\Z)$ (cf. Lemma 1 at \cite{ghys_godbillon-vey}). The Birkhoff section actually exhibits a path of Fried surgeries connecting $(\phi,M)$ with the suspension flow $(\phi^A,M_A)$:
$$[(\phi,M)]\xrightarrow[]{(\gamma_1,-m_1),\dots,(\gamma_k,-m_k)}[(\phi^A,M_A)].$$
Observe that these surgeries are necessarily integral, and the multiplicity $m_i=m(\gamma_i,\Sigma)\in\Z$ of the Birkhoff section at each $\gamma_i$ has opposite sign and same modulus as the surgery slope on $\gamma_i$.

\begin{defn}
\label{defn_per_Z}
Given a periodic orbit $\alpha$ of the suspension flow $(\phi^A,M_A)$ define $\per_\Z(\alpha)$ to equal the pairing $[\alpha]\cdot[\T^2]$ of the homology class $[\alpha]\in H_1(M_A)$ with a generator $[\T^2]$ of $H^1(M_A;\Z)$.
\end{defn}

\begin{thmB}
Let $A\in\SL(2,\Z)$ be a matrix with $\tr(A)\geq 3$ satisfying that $A$ is conjugated to $A^{-1}$ in $\GL(2,\Z)$. Then, given any periodic orbit $\gamma$ of the suspension flow $(\phi^A,M_A)$ there exists $m_0=m_0(\gamma)>0$ such that, for every integer $m$ with $|m|\geq m_0$, there are infinitely many Birkhoff sections $\iota:(\Sigma,\partial\Sigma)\to(M_A,\Gamma)$ verifying that:
\begin{enumerate}
\item $\Gamma$ consists of two periodic orbits $\gamma$ and $\alpha$, 
\item The multiplicities and number of components on the boundary are:
	\begin{itemize}
	\item[$\bullet$] $m(\gamma,\Sigma)=-m$, $m(\alpha,\Sigma)=m$,
	\item[$\bullet$] $p(\gamma,\Sigma)=\per_\Z(\alpha)$, $p(\alpha,\Sigma)=\per_\Z(\gamma)$,
	\end{itemize}	   
\item $gen(\Sigma)=1$ and the first return map $P_\Sigma$ is conjugated to $A$.
\end{enumerate}
\end{thmB}

%%%%%%%%%%%%%%%%%%%%%%%%%%%%%%%%%%%%%%%%%%%%%%%%%%%%%%%%%%%%%%%%%%%%%%%%%%%%%%%%%%%%%%%%%%%%%%%%%%%%%%%%%%%%%%%%%%%%%%%%%%%%%%%%%%%%%%%%%%%%%%%%%%%
\subsection{Construction of the paths}
Let's prove Theorem \textbf{A}-\textbf{A'} and Proposition \textbf{B}.
For this, fix a matrix $A\in\SL(2,\Z)$ with $\tr(A)\geq 3$, choose some periodic orbit $\gamma$ of the suspension flow $\{\phi_t^A:M_A\to M_A\}_{t\in\R}$, let $m>0$ be an integer and define $(\psi,N)$ to be the flow obtained from the suspension $(\phi^A,M_A)$ by doing a Fried surgery along $\gamma$ with slope equal to $m$. We denote by $\Pi:\textsl{Per}(\phi^A,M_A)\to\textsl{Per}(\psi,N)$ the natural bijection between periodic orbits induced by this surgery, and denote $\delta=\Pi(\gamma)$.

\begin{proof}[Proof of Theorem \textbf{A}-\textbf{A'}]
We have that:
\begin{itemize}
\item By Proposition \ref{prop_R-cov_Fried_surgery}, for every $m>0$ then $(\psi,N)$ is a skewed $\R$-covered Anosov flow, positively twisted;

\item By Proposition \ref{prop_hyperbolic_Fried_surgery}, exists $m_0=m_0(\gamma)$ such that $N$ is hyperbolic if $m\geq m_0$;

\end{itemize}
Setting $m\geq m_0$ we get a skewed $\R$-covered Anosov flow $\{\psi_t:N\to N\}_{t\in\R}$ on a hyperbolic 3-manifold. By Theorem \ref{thm_free_homot_classes}, there exists an infinite set $\mathcal{FH}(\delta)$ of $\psi$-periodic orbits, each of them orientation preserving isotopic to $\delta$. 

\begin{figure}[t]
\centering
\includegraphics[width=\textwidth]{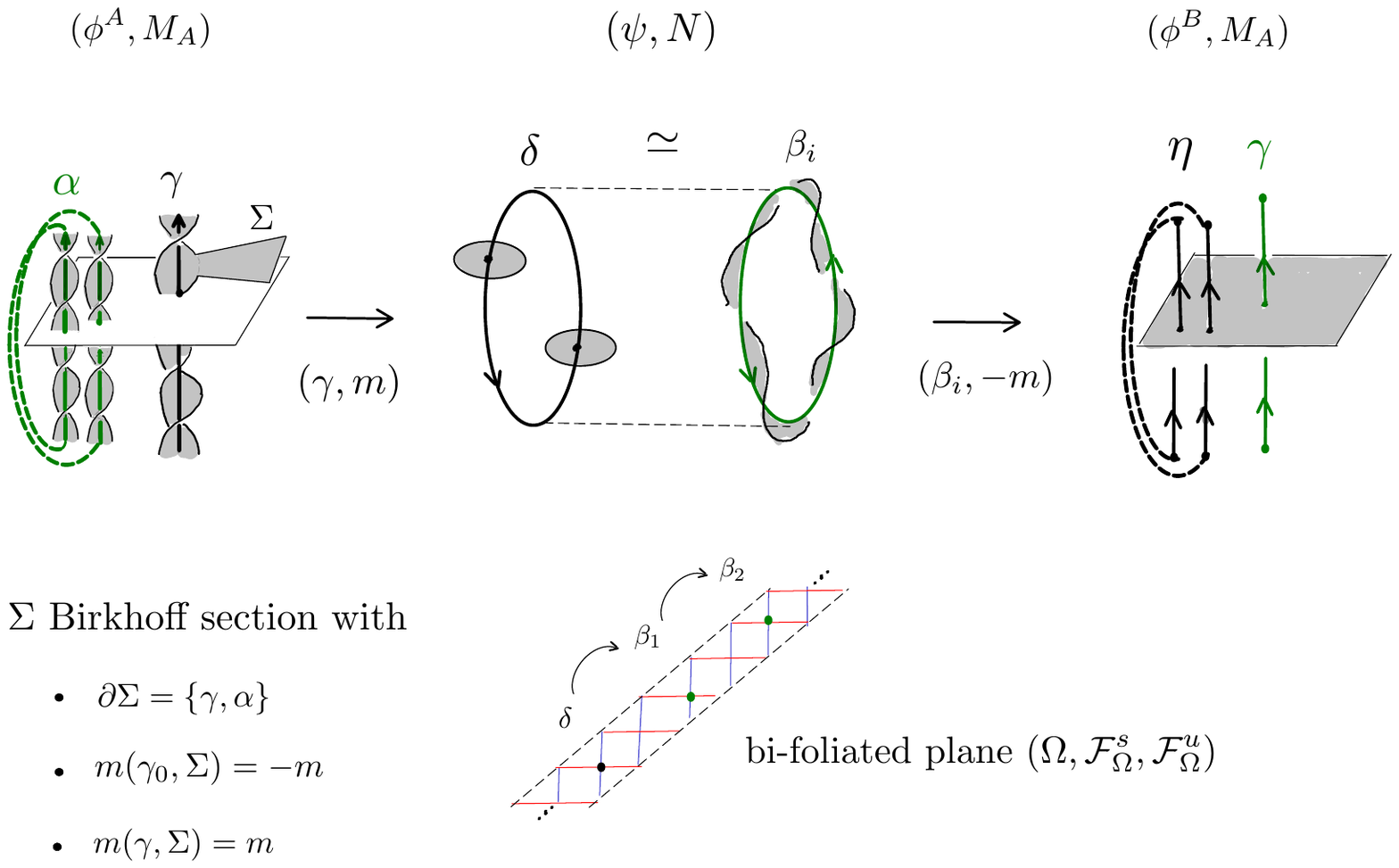}
\caption{Surgery $(\phi^A,M_A)\xrightarrow[]{(\gamma,m),(\alpha,-m)}(\phi^{B},M_A)$, $B\simeq A$ or $B\simeq A^{-1}$.}
\label{fig_infty-many_closed_paths}
\end{figure}

Consider one orbit $\beta\in\mathcal{FH}(\delta)$ different from $\delta$. Since $\delta$ and $\beta$ are equivalent as oriented knots, in particular there exists a homeomorphism on $N$ taking $\delta\mapsto\beta$ and preserving their orientations, so the surgery on $(\psi,N)$ with parameters $(\beta,-m)$ will produce an Anosov flow on a manifold homeomorphic to $M_A$. Recall that $M_A$ is a sol-manifold, therefore from the classification stated in Theorem \ref{thm_classific_sol-mflds} and Proposition \ref{prop_rigidity} this flow must be orbitally equivalent to the suspension of either $A$ or $A^{-1}$. That is:
$$[(\phi^A,M_A)]\xmapsto[]{(\gamma,m)}[(\psi,N)]\xmapsto[]{(\beta,-m)}[(\phi^B,M_A)],\ \text{where}\ B\simeq A\ \text{or}\ B\simeq A^{-1}.$$

If we assume in addition that $A$ is conjugated to $A^{-1}$ in $GL(2,\Z)$ then the surgery described above is a closed loop in $\mathbf{G}_\circ$ with base point at the class of $(\phi^A,M_A)$. Define 
\begin{itemize}
\item $\alpha$ to be the preimage of $\beta$ under the surgery $(\phi^A,M_A)\xmapsto[]{(\gamma,m)}(\psi,N)$,
\item $\eta$ to be the image of $\delta$ under the surgery $(\psi,N)\xmapsto[]{(\beta,-m)}(\phi^{A},M_A)$.
\end{itemize}
Let $\gamma'$ be the image of $\beta$ under the surgery $(\psi,N)\xmapsto[]{(\beta,-m)}(\phi^{A},M_A)$. Along the construction there is a sequence of homeomorphisms 
$M_A\setminus\gamma\to N\setminus\delta\to N\setminus\beta\to M_A\setminus \gamma'$  
so by Proposition \ref{prop_rigidity2} we have that, up to a self-orbital equivalence of ($\phi^A,M_A)$, the orbit $\gamma'$ is equal to $\gamma$. We have codified paths:
\begin{align}
\label{eq_path1}
& [(\phi^{A},M_A)]\xmapsto[]{(\gamma,m),(\eta,-m)}[(\phi^A,M_A)]\ \text{taking}\ \Pi:\eta\mapsto\gamma\ \text{and}\ \Pi:\gamma\mapsto\alpha\\
\label{eq_path2}
& [(\phi^{A},M_A)]\xmapsto[]{(\gamma,m),(\alpha,-m)}[(\phi^A,M_A)]\ \text{taking}\ \Pi:\alpha\mapsto\gamma\ \text{and}\ \Pi:\gamma\mapsto\eta.
\end{align}

The global transverse section of $\{\phi_t^{A}:M_A\to M_A\}_{t\in\R}$ is transformed by the surgery \eqref{eq_path1} into the desired Birkhoff section $\iota:(\Sigma,\partial\Sigma)\to(M_A,\Gamma)$ for $(\phi^A,M_A)$, where $\Gamma=\{\gamma,\alpha\}$. In fact: 
\begin{itemize}
\item It is direct that $m(\gamma,\Sigma)=-m$ and $m(\alpha,\Sigma)=m$;

\item It is direct that $\Sigma$ has $p(\alpha,\Sigma)=\per_\Z(\gamma)$ boundary components along $\alpha$;

\item To determine $p(\gamma,\Sigma)$, since $[\iota(\partial\Sigma)]=0\in H_1(M_A)\simeq\Z$ we have the relation
$$0=p(\gamma,\Sigma)\cdot\per_\Z(\gamma)\cdot m(\gamma,\Sigma)+p(\alpha,\Sigma)\cdot\per_\Z(\alpha)\cdot m(\alpha,\Sigma),$$
from where we conclude that $p(\gamma,\Sigma)=\per_\Z(\alpha)$.   
\end{itemize} 

Finally, the set $\mathcal{Q}(m,\gamma)$ consists in all periodic orbits $\alpha$ of $(\phi^A,M_A)$ such that $\beta=\Pi(\alpha)$ belongs to $\mathcal{FH}(\Pi(\gamma))$. \qedhere
\end{proof}

\begin{proof}[Proof of Proposition \textbf{B}]
Proposition \textbf{B} follows from an estimation given in \cite{Barthelme-Fenley_II}, on the total number of periodic orbits $\beta$ in $\mathcal{FH}(\delta)$ with period $\per(\beta)\leq t$, for every $t\geq 0$.

\begin{thm}[Theorem G of \cite{Barthelme-Fenley_II}]
If $\{\psi_t:N\to N\}_{t\in\R}$ is a $\R$-covered Anosov flow in a hyperbolic manifold then there exist constants $C_0,t_0>0$ such that, for every $\psi$-periodic orbit $\delta$, it is verified that 
\begin{equation}
\label{eq_1}
|\{\beta\in\mathcal{FH}(\delta):\per(\beta)\leq t\}|\leq C_0\sqrt{t}\log(t),\ \forall\ t>t_0.
\end{equation}
\end{thm}

To prove Proposition \textbf{B} it is enough to show that there exist $C_1,t_1>0$ such that
\begin{equation}
\label{eq_2}
|\{\alpha\in\mathcal{Q}(\gamma,m):\per(\alpha)\leq t\}|\leq C_1\sqrt{t}\log(t),\ \forall\ t>t_1. 
\end{equation}
Because, in this case, it follows that 
$$\lim_{t\to\infty}\frac{\left\vert\left\{\alpha\in\mathcal{Q}(\gamma,m):\per(\alpha)<t\right\}\right\vert}{\left\vert\left\{\alpha\in\PP:\per(\alpha)<t\right\}\right\vert}\leq \lim_{t\to\infty}\frac{C_1\sqrt{t}\log(t)}{\kappa\cdot te^{\kappa t}}=0
$$ 
Here we use the fact that, for some constant $\kappa>0$, the number of periodic orbits with period less than $t$ is equivalent to $\kappa\cdot te^{\kappa t}$ as $t\to+\infty$, as provided by the Bowen-Margulis theorem (see, e.g. \cite{Katok-Hasselblatt}).

To derive \eqref{eq_2} from \eqref{eq_1} we have to compare the periods of a $\phi^A$-periodic orbit $\alpha$ and the corresponding $\psi$-periodic orbit $\beta=\Pi(\alpha)$. To do this, let $T$ be a global transverse section of $(\phi^A,M_A)$. The surgery 
$$(\phi^A,M_A) \xrightarrow[]{(\gamma,m)}(\psi,N)$$
transforms the global transverse section $T$ into a Birkhoff section for $(\psi,N)$ that we denote by $\iota:(S,\partial S)\to(N,\{\delta\})$. It has $p(S,\delta)=\per_\Z(\gamma)$ boundary components (cf. Definition \ref{defn_per_Z}) attached along $\delta$ with multiplicity $m$. It is direct to check that every $\phi^A$-periodic orbit $\alpha\neq\gamma$ is sent to a $\psi$-periodic orbit $\beta$ transverse to the Birkhoff section $S$ satisfying $[\beta]\cdot[S]=\per_\Z(\alpha)$, where this last expression indicates pairing of relative homology classes in $H_*(N;\delta)$. Furthermore, since the manifolds supporting the flows $\phi^A$ and $\psi$ are compact, there exists constants $0<a_1<a_2$ and $0<b_1<b_2$ such that 
\begin{align*}
& a_1\cdot\per(\alpha)\leq\length(\alpha)\leq a_2\cdot\per(\alpha),\ \text{for every}\ \alpha\in\PP\\
& b_1\cdot\per(\beta)\leq\length(\beta)\leq b_2\cdot\per(\beta),\ \text{for every}\ \beta\in\textrm{Per}(\psi,N).  
\end{align*}

We have that:
\begin{itemize}
\item Since the $\phi^A$-orbits are transverse to $T$ there exists $l_1>0$ such that every $\phi^A$-orbit segment of length at most $l_1$ intersects at most once the transverse section $T$. Hence, every periodic orbit $\alpha$ with length bigger than $l_1$ has a number of intersections with $T$ bounded from above by $(a_2/l_1)\cdot\per(\alpha)$. It follows that: 
$$\per(\alpha)\geq\kappa_1\cdot\per_\Z(\alpha),\ \text{for some constant}\ \kappa_1>0.$$ 

\item There exists $l_2>0$ such that every $\psi$-orbit segment of length bigger that $\tau_2$ has at least one intersection with the Birkhoff section $S$. Hence, every $\psi$-periodic orbit $\beta$ with length bigger that $l_2$ and transverse to the interior of $S$ must intersect this surface in at least $(b_1/l_2)\cdot\per(\beta)$ points. It follows that: 
$$[\beta]\cdot[S]\geq\kappa_2\cdot\per(\beta),\ \text{for some constant}\ \kappa_2>0.$$ 
\end{itemize}
In addition, since every Anosov flow has at most finitely many periodic orbits of length (or period) lower than any given constant, there exists $t_*>0$ such that $\length(\alpha)>l_1$ and $\length(\beta)>l_2$ whenever $\per(\alpha)>t_*$. Hence, from the items above we see that there exists $\kappa_3>0$ such that $\per(\alpha)\geq\kappa_3\cdot\per(\beta)$, whenever $\per(\alpha)>t_*$. This implies 
$$\{\alpha\in\mathcal{Q}(\gamma,m):t_*<\per(\alpha)\leq t\}\subset\{\Pi^{-1}(\beta):\beta\in\mathcal{FH}(\delta)\text{ and }\per(\beta)\leq t/\kappa_3\},\ \forall\ t>t_*.$$
Finally, by setting $C_*=|\{\alpha\in\mathcal{Q}(\gamma,m):\per(\alpha)\leq t_*\}|$, there exists constants $C_1>C_0\sqrt{1/\kappa_3}$ and $t_1>\max\{t_0,t_*\}$ such that 
$$|\{\alpha\in\mathcal{Q}(\gamma,m):\per(\alpha)\leq t\}|\leq C_*+C_0\sqrt{t/\kappa_3}\log(t/\kappa_3)\leq C_1\sqrt{t}\log(t),\ \forall\ t>t_1.\qedhere$$
\end{proof}

\section*{Acknowledgements}
We want to thank S. Fenley for useful comments and discussions on the present result.

%%%%%%%%%%%%%%%%%%%%%%%%%%%%%%%%%%%%%%%%%%%%%%%%%%%%%%%%%%%%%%%%%%%%%%%%%%%%%%%%%%%%%%%%%%%%%%%%%%%%%%%%%%%%%%%%%%%%%%%%%%%%%%%%%
\bibliographystyle{plain}

%%%%%%%%%%%%%%%%%%%%%%%%%%%%%%%%%%%%%%%%%%%%%%%%%%%%%%%%%%%%%%%%%%%%%%%%%%%%%%%%%%%%%%%%%%%%%%%%%%%%%%%%%%%%%%%%%%%%%%%%%%%%%%%%%
\end{document}